\documentclass{article}[10pt]
\usepackage{amssymb}
\usepackage{amsmath, amsthm}
\usepackage{color}
\newtheorem{theorem}{Theorem}[section]
\newtheorem{lemma}{Lemma}[section]
\newtheorem{corollary}{Corollary}[section]
\newtheorem{definition}{Definition}[section]
\newtheorem{proposition}{Proposition}[section]
\newtheorem{remark}{Remark}[section]

\begin{document}
\title
{Consistency of the Local Density Approximation
and Generalized Quantum Corrections
for Time Dependent Closed Quantum
Systems}
\author{Joseph W. Jerome\footnotemark[1]}
\date{}
\maketitle
\pagestyle{myheadings}
\markright{Quantum Corrections}
\medskip

\vfill
{\parindent 0cm}

{\bf 2010 AMS classification numbers:} 
35Q41; 81Q05.  


\fnsymbol{footnote}
\footnotetext[1]{Department of Mathematics, Northwestern University,
                Evanston, IL 60208. \\In residence:
George Washington University,
Washington, D.C. 20052}

\begin{abstract}
Time dependent quantum systems are the subject of intense inquiry, in
mathematics, science, and engineering,  
particularly at the atomic and molecular levels.
In 1984, Runge and Gross introduced time dependent density functional
theory (TDDFT), a non-interacting electron model, which predicts
charge exactly. An exchange-correlation potential is included in the
Hamiltonian to enforce this property. 
We have previously investigated such systems
on bounded domains for Kohn-Sham potentials by use of evolution operators
and fixed point theorems.  
In this article, motivated by usage in the physics community,  we consider
local density approximations (LDA) for building 
the exchange-correlation potential,
as part of a set of quantum corrections. 
Existence and uniqueness of solutions are established separately 
within a framework for general quantum corrections,
including time-history corrections and ionic Coulomb potentials,
in addition to LDA potentials.
In summary, we are able to demonstrate a unique weak solution, 
on an  arbitrary time
interval, for a general class of quantum corrections, including 
those typically used in numerical simulations of the model.
\end{abstract}

\section{Introduction}
Time dependent density functional theory (TDDFT) was introduced
by E.~Runge and E.K.U.~Gross in \cite{RG} as a non-interacting electron
model which tracks electron charge exactly. An 
exposition of the subject may be found in \cite{U}.
When Kohn-Sham potentials are used, the electronic Hamiltonian includes any
(time dependent) external potentials, ionic potentials, the Hartree
potential, and the compensating exchange-correlation potential to ensure 
the non-interacting and charge exactness features of the model.
By permitting time dependent potentials,
TDDFT extends the nonlinear Schr\"{o}dinger equation, 
which has been studied extensively \cite{CH,Caz}, principally with
potentials not directly depending on time. Some progress for time
dependent linear Hamiltonians has been made \cite{MR}. 
In previous work \cite{JP,J1}, we analyzed  
closed quantum systems on bounded domains of
${\mathbb R}^{3}$ 
via time-ordered evolution operators. The article \cite{JP} demonstrated
strong $H^{2}$ solutions, compatible with simulation, whereas  
the article \cite{J1} 
demonstrated weak 
solutions; 
\cite{J1} also includes the exchange-correlation component of the 
Hamiltonian potential, not included in \cite{JP}, which 
is a nonlocal time-history term, satisfying certain  regularity
hypotheses.
TDDFT is a significant field for applications, 
including computational nano-electronics and chemical
physics \cite{tddft}. 

An important early article in the time dependent case, directed toward 
Hartree-Fock Hamiltonians, is 
\cite{CLB}. This article included nuclear dynamics as a coupled classical
dynamical system, and defined an electronic Hamiltonian in terms of a
kinetic term, together with a Hartree potential, an ionic potential with
mobile point masses, and an external, electric-field-induced potential.  
The mathematical framework was defined on ${\mathbb R}^{3}$ in terms of a
Cauchy problem with $H^{2}$ initial datum.
A recent
article directed toward TDDFT, 
in which a quantum correction is of
local density type, is \cite{SCB}; this article couples quantum mechanics 
and control theory. Neither of these articles allows for a time-history
exchange-correlation potential.

In this article, we introduce a class of quantum corrections, including 
the local density approximation, but also ionic Coulomb potentials and
time-history potentials.
As we demonstrate below, 
smoothing of such potentials provides a model within the
framework of \cite{J1}. By using compactness arguments suggested in
\cite{Caz}, we are able to obtain a solution of the originally posed
model. 
Uniqueness is also established.
The use of evolution
operators and smoothing as presented here is consistent with techniques in
the applied literature \cite{tddft} and provides direct support for
successive approximation and other numerical procedures \cite{JF,CP}. 
In this sense, the results of this article are more inclusive than an
existence/uniqueness analysis.

In the following subsections of the introduction, we 
summarize the basic results of \cite{J1}, 
as a starting point for the present article.
In section two, we formulate the new model, which incorporates the 
category of quantum corrections,
and we prove that its smoothed version lies
within the scope of \cite{J1}. In section three, we introduce the
compactness arguments, and establish existence of a weak solution as the
limit of solutions of the smoothed model. 
Uniqueness is established in section four. We conclude with some summary
remarks.
\subsection{The model}
\label{origmodel}
In its original form, without ionic influence,  
TDDFT includes three components for the electronic potential:  
an external potential, the Hartree potential, 
and a general non-local term representing the exchange-correlation potential, 
which is assumed to include a time-history part.
If $\hat H$ denotes
the Hamiltonian operator of the system, then the state $\Psi(t)$ of the
system obeys the nonlinear Schr\"{o}dinger equation,
\begin{equation}
\label{eeq}
i \hbar \frac{\partial \Psi(t)}{\partial t} = \hat H \Psi(t).
\end{equation}
Here, 
$\Psi = \{\psi_{1}, \dots, \psi_{N}\}$ 
consists of 
$N$ 
orbitals, and the charge density 
$\rho$ 
is defined by 
$$ \rho({\bf x}, t) = |\Psi({\bf x}, t)|^{2} = 
\sum_{k = 1}^{N} |\psi_{k} ({\bf x}, t)|^{2}.
$$ 
An initial condition,
\begin{equation}
\label{ic}
\Psi(0) = \Psi_{0}, 
\end{equation}
and boundary conditions are included. 
The particles are confined to a bounded Lipschitz region 
$\Omega \subset {\mathbb R}^{3}$ 
and homogeneous Dirichlet boundary conditions hold 
within a closed system. 
$\Psi$ 
denotes a finite vector function of space and time. 
The effective potential 
$V_{\rm e}$ 
is a  real scalar function of the form,
$$
V_{\rm e} ({\bf x},t, \rho) = V({\bf x}, t) + 
W \ast \rho + \Phi({\bf x}, t, \rho).
$$
Here, 
$W({\bf x}) = 1/|{\bf x}|$ 
and the convolution
$W \ast \rho$ 
denotes the Hartree potential. If $\rho$ is extended as zero outside
$\Omega$, then, for ${\bf x} \in \Omega$, 
$$
W \ast \rho \; ({\bf x})=\int_{{\mathbb R}^{3}} 
W({\bf x} -{\bf y}) \rho({\bf y})\;d {\bf y},
$$
which depends only upon values $W({\bf z})$,
$\|{\bf z}\|\leq 
\mbox{diam}(\Omega)$. We may redefine $W$ 
smoothly outside this set,
so as to obtain a function of compact support for which Young's inequality
applies. The exchange-correlation potential 
$\Phi$ 
represents a time-history of $\rho$:
$$
\Phi({\bf x}, t, \rho)= \Phi({\bf x}, 0, \rho) + 
\int_{0}^{t} \phi({\bf x}, s, \rho) \; ds.
$$  
The Hamiltonian operator is given by, 
\begin{equation}
\hat H  
= -\frac{\hbar^{2}}{2m} \nabla^{2} 
 +V({\bf x}, t) + 
W \ast \rho + \Phi({\bf x}, t, \rho),
\label{Hamiltonian1}
\end{equation}
and 
$m$ 
designates the effective mass and $\hbar$ the  
normalized Planck's constant.
If ionic influence is present, then (\ref{Hamiltonian1}) is adjusted,
typically by Coulomb potentials.
\subsection{Definition of weak solution and function
spaces}
The solution 
$\Psi$ 
is continuous from the time interval 
$J$, 
to be
defined shortly,  
into the finite energy Sobolev space
of complex-valued 
vector functions which vanish in a generalized sense on the boundary, 
denoted 
$H^{1}_{0}(\Omega)$: $\Psi \in C(J; H^{1}_{0})$. 
The time derivative is continuous from 
$J$ 
into the dual 
$H^{-1}$ 
of 
$H^{1}_{0}$: 
$\Psi \in C^{1}(J; H^{-1})$.  
The spatially dependent test functions 
$\zeta$ 
are
arbitrary in 
$H^{1}_{0}$. 
The duality bracket is denoted 
$\langle f, \zeta \rangle$. 
{\it Norms and inner products are discussed in Appendix \ref{appendixA}.}
We will make use of the equivalence of the standard $H^{1}_{0}$ norm and
the gradient seminorm, due to the Poincar\'{e} inequality, 
which holds for bounded domains $\Omega$ \cite{Leoni}.
\begin{definition}
\label{weaksolution}
For 
$J=[0,T]$,  
the vector-valued function 
$\Psi = \Psi({\bf x}, t)$ 
is a  
weak solution of (\ref{eeq}, \ref{ic}, \ref{Hamiltonian1}) if 
$\Psi \in C(J; H^{1}_{0}(\Omega)) \cap C^{1}(J;
H^{-1}(\Omega)),$ 
if 
$\Psi$ 
satisfies the initial condition 
(\ref{ic}) for 
$\Psi_{0} \in H^{1}_{0}(\Omega)$, 
and if 
$\forall \; 0 < t \leq T$: 
\vspace{.25in}
\begin{equation}
i \hbar\langle \frac{\partial\Psi(t)}{\partial t},
\zeta \rangle  = 
\int_{\Omega} \frac{{\hbar}^{2}}{2m}
\nabla \Psi({\bf x}, t)\cdotp \nabla { \zeta}({\bf x}) 
+ V_{\rm e}({\bf x},t,\rho) \Psi({\bf x},t) { \zeta}({\bf x})
d{\bf x}. 
\label{wsol}
\end{equation}
\end{definition}
\subsection{Hypotheses and theorem statement}
\label{hyps}
We provide some discussion, relevant to the physical model, prior to the
statement of the hypotheses. Additional discussion will be provided
following the hypotheses. It is emphasized that the hypotheses of this
subsection are those required for the original theory of \cite{J1} to apply;
this was accomplished with evolution operators and the Banach fixed point
mapping. Subsequent sections of this article consider more general families
of correction potentials. 

The time-history potential $\Phi({\bf x}, t, \rho)$ 
above has a structure, including the time-integrated part, which is
motivated by \cite[Eqs.\ (15), (17)]{MarGross}. This article characterizes
the action functionals $A$ whose variational derivatives with respect to 
$\rho$ yield appropriate exchange-correlation potentials.  
The form of $\Phi$ selected above represents a general statement of these
ideas. It is not unreasonable that the mathematical hypotheses, to be
stated shortly, should resemble 
the known properties of the Hartree potential because of
the restorative nature of exchange and correlation. From a mathematical
perspective, the model permits multiple `copies' of $\Phi$, allowing for
quantum corrections. These are seen to be important for applications. For
example, in the quantum chemistry community \cite{SaddTeter}, it is
appropriate to split $\Phi$: the exchange part is represented by a
weighted density approximation (WDA), while the correlation part is
represented by a local density approximation (LDA). 
The nonlocal WDA form for $\Phi$ is appropriate for nonuniform mixtures
\cite{DL}. The general form we have allowed for $\Phi$ is intended to
anticipate applications of this type.

The following
hypotheses are those for which the evolution operator theory of \cite{J1}
applies. The present article builds upon this established theory.  

We assume the following 
hypotheses in order to apply the results of
\cite{J1}.
\begin{itemize}
\item
\begin{enumerate}
\item
The time-history potential 
$\Phi$ 
 is continuous in 
$t \in J$  
into 
$H^{1}_{0}$.
\item
$\Phi$ is bounded, uniformly  in 
$t \in J$, 
from $H^{1}_{0}$
into 
$W^{1,3}$.
More precisely,
by boundedness, we mean that the family $\{\Phi(\cdotp, t, \cdotp)\}$  
maps every fixed ball in $H^{1}_{0}$ 
into a fixed ball in 
$W^{1,3}$,
uniformly in
$t$.
\end{enumerate}
\item
The derivative 
$\partial \Phi/\partial t = 
\phi$ 
is assumed measurable, and bounded in its arguments. 
\item
Furthermore, the following smoothing condition
is assumed, 
expressed by a (uniform) Lipschitz norm condition:
$$\forall t \in [0,T],
\mbox{\rm if} \; 
\|\Psi_{j}\|_{H^{1}_{0}},
j=1,2, \; \mbox{are bounded by} \;  r,
$$
then
\begin{equation}
\|[\Phi(\cdotp,t,|\Psi_{1}|^{2})-\Phi(\cdotp,t,|\Psi_{2}|^{2})]\psi\|_{H^{1}}
\leq 
C(r) \|\Psi_{1} - \Psi_{2}\|_{H^{1}_{0}} \|\psi\|_{H^{1}_{0}}. 
\label{ecfollowsH}
\end{equation}
Here, 
$\psi$ 
is arbitrary in 
$H^{1}_{0}$ 
and 
$C(r)$ 
depends only on 
$r$.
\item
If $\Phi(\cdotp, 0, \rho)$ fails to be a nonnegative functional of
$\rho = |\Psi|^{2}$, 
we assume that it
satisfies, uniformly in $t$,  
for $\; \|\Psi(t)\|_{L^{2}} =
\|\Psi_{0}\|_{L^{2}}$, 
the constraint that 
\begin{equation}
\label{constraint}
\|\Phi(\cdotp, 0,
|\Psi|^{2})
|\Psi|^{2}\|_{L^{1}}   \leq C_{1}  \|\nabla \Psi\|_{L^{2}}^{2} 
+ C_{2}, \; \Psi(t) \in H^{1}_{0}, 
\end{equation}
for nonnegative constants $C_{1}$ and $C_{2}$. It is required that 
$C_{2}$ depend only on $\|\Psi_{0}\|_{L^{2}}$ and the problem data,  
and $C_{1}$ is
sufficiently small:
\begin{equation}
\label{sufsmall}
C_{1} < \frac{\hbar^{2}}{2m}. 
\end{equation} 
\item
The so-called external potential 
$V$ 
is assumed to be 
continuously
differentiable on the closure of the space-time domain. 
\end{itemize}
\begin{remark}
We comment here on the hypotheses. 
\begin{enumerate}
\item
The regularity assumed for $\Phi$ in the first assumption 
is consistent with certain
requirements of TDDFT. One of these is the Zero Force Theorem
\cite{U}, which imposes a gradient condition on $\Phi$. 
We note that the Hartree potential satisfies these conditions. In fact,
any convolution of the form $\Phi = F \ast \rho$, where $F \in W^{1,1}$,
satisfies the conditions. 
\item
An inequality of the form 
(\ref{ecfollowsH}) 
is satisfied by the Hartree
potential \cite[Theorem 3.1]{JNA}, 
and by any convolution of the form $\Phi = F \ast \rho$, with
$F \in L^{2}$ and $\nabla F \in L^{1}$.
It was used in \cite{J1} to
construct the contraction mapping used there for the evolution operator. 
For quantum corrections not
satisfying this condition, the smoothing is utilized in the following
section in order to place the smoothed systems within this framework.
\item
Hypotheses (\ref{constraint}, \ref{sufsmall}) are relevant only when the
associated potentials are negative. This is expected to occur for
restoring potentials and certain Coulomb potentials. 
In the following section, it will be necessary to
smooth certain components of the quantum correction potential. The
smoothed Coulomb potentials satisfy 
(\ref{constraint}, \ref{sufsmall}) without qualification. However, for 
smoothed LDA approximations,  
there is a disparity in exponent bounds for $\alpha$. 
A smaller range is necessary for negative potentials (see (\ref{proplamb}) to
follow for verification in this case).
Also, unsmoothed convolutions of the form $\Phi = F \ast \rho$, with
$\nabla F \in L^{1}$, satisfy the conditions if they have sufficiently
small $L^{\infty}$ bounds.
\end{enumerate}
\end{remark}
The following theorem was proved in \cite{J1}, based upon the evolution
operator as presented in \cite{J2}, and will provide a solution
for the smoothed problem on 
$J$ 
as introduced in the following section.
\begin{theorem}
\label{EU}
For any interval 
$[0,T]$, 
the system (\ref{wsol}) in Definition
\ref{weaksolution}, 
with Hamiltonian
defined by (\ref{Hamiltonian1}),  
has a unique weak solution if the hypotheses of section \ref{hyps} hold. 
\end{theorem}
\section{Quantum Corrections and the Local Density Approximation}
\label{qcsection}
In this section, we define a class of quantum correction potentials,   
including the local density approximation to the
exchange-correlation potential 
$\Phi$. These correction potentials are of three types.
\begin{enumerate}
\item
The local density approximation, discussed in Definition \ref{Def2.1}
to follow.
This potential is designated as $\Phi_{\mbox{\rm lda}}(\rho)$.
\item
A finite number of Coulomb ionic potentials, $c_{j} W(\cdotp - {\bf
x}_{j})$, subject to the Born-Oppenheimer approximation. In particular, 
the ionic masses are assumed to be point masses, at fixed locations ${\bf
x}_{j} \in \Omega$. The function $W$ is introduced in section
\ref{origmodel}. The constants $c_{j}$ may be positive or negative.
The aggregate of these Coulomb potentials is designated $\Phi_{\mbox{\rm
c}}(\cdotp)$.
\item
A time-history potential of the structure of $\Phi$, introduced in section
\ref{origmodel}. The presence of this potential allows for physical
modeling flexibility, since the exchange potential and the correlation
potential are viewed separately in TDDFT. We permit one of these to be
approximated locally and the other by a time-history among the modeling
choices. We retain the notation $\Phi(\cdotp, t, \rho)$ for this
component, assumed to satisfy the hypotheses detailed in section  
\ref{hyps}. Also, it is assumed that $\Phi(\cdotp, t,
\rho_{n}(\cdotp, t))$ converges in $L^{2}$, uniformly in $t$, if 
$\rho_{n}(\cdotp, t)$ converges in $L^{2}$, uniformly in $t$.
\end{enumerate}
The consolidated quantum correction potential is then given by
\begin{equation}
\Phi_{\mbox{\rm qc}}(\cdotp, t, \rho) = \Phi_{\mbox{\rm lda}}(
\rho) + \Phi_{\mbox{\rm c}}(\cdotp) + 
\Phi(\cdotp, t, \rho).
\end{equation}
\begin{definition}
\label{Def2.1}
The local density approximation $\Phi_{\mbox{\rm lda}}$ is 
now defined.
We consider the following approximation,
where 
$\lambda$ 
is a real constant, positive or negative.
\begin{equation}
\label{lda}
\Phi_{\rm lda}(\rho) = \lambda \rho^{\alpha/2} = 
\lambda |\Psi|^{\alpha}.
\end{equation} 
Additionally,
\begin{itemize}
\item
If 
$\lambda > 0$, 
the range of 
$\alpha$ 
is 
$1 \leq \alpha < 4$.
\item
If 
$\lambda < 0$, 
the range of 
$\alpha$ 
is 
$1 \leq \alpha \leq 4/3$.
Also, $|\lambda|$ must be sufficiently small, consistent with
(\ref{constraint}) and (\ref{sufsmall}).
\end{itemize}
\end{definition}
We redefine the Hamiltonian 
considered here as
\begin{equation*}
\hat H  
 =  -\frac{\hbar^{2}}{2m} \nabla^{2} 
 +V({\bf x}, t) + 
W \ast \rho + \Phi_{\rm qc}(\cdotp, t, \rho), 
\end{equation*}
\begin{equation}
\Phi_{\rm qc}(\cdotp, t, \rho)  =  
\underbrace{\lambda |\Psi|^{\alpha}(\cdotp, t)}_{\Phi_{\rm lda}} 
+ \underbrace{\sum_{j=1}^{M} c_{j}\frac{1}{|\cdotp - {\bf
x}_{j}|}}_{\Phi_{\rm c}} 
+ \Phi(\cdotp, t,
\rho). 
\label{Hamiltonian2}
\end{equation}
The proofs accommodate a finite number of terms in
$\Phi_{\rm lda}$. One term has been chosen for simplicity.
The parameters of $\Phi_{\rm lda}$ satisfy the assumptions of Definition
\ref{Def2.1}. The numerical constants $c_{j}$ are of arbitrary sign,
and the ionic locations ${\bf x}_{j}$ are fixed interior points in
$\Omega$.
$\Phi$ satisfies the hypotheses 
specified in (3) above, 
and is a nonlocal potential such as weighted density approximation.
Convolutions, discussed in Remark 1, represent an important class. 
{\color{blue}{
For simplicity, we assume that the leading part, $\Phi(\cdotp, 0, \rho)$,
is conserved, up to a positive constant multiple. This holds for
convolutions and other important examples.}}  
The time integrated part of $\Phi$ is motivated by \cite{MarGross}.
The following theorem is the goal of our analysis. 
\begin{theorem}
\label{central}
If the effective potential is redefined by 
\begin{equation}
\label{redefined}
V_{\rm e}({\bf x}, t, \rho) =
 V({\bf x}, t) + 
W \ast \rho + \Phi_{\rm qc}(\cdotp, t, \rho),
\end{equation}
then there is a weak solution of (\ref{wsol})
in the regularity class 
$C(J; H^{1}_{0}) \cap C^{1}(J;H^{-1})$ 
which satisfies the specified initial condition.
{\color{blue}{
Uniqueness holds except possibly for $1 < \alpha < 2$.}}
\end{theorem}
The existence part of the proof of Theorem \ref{central} 
is carried out in section three (see
Theorems \ref{central1} and \ref{central2}). 
The uniqueness is demonstrated in section four.
\subsection{The smoothing}
\label{smoothing}
We begin by defining a standard convolution \cite{LiebLoss}. 
\begin{definition}
\label{convolution}
Suppose that
a nonnegative function 
$\phi_{1}$ 
is given, 
$\phi_{1} \in C^{\infty}_{0}({\mathbb R}^{3})$, 
of integral one. Set
$$
\phi_{\epsilon}({\bf x}) = 
\epsilon^{-3}\phi_{1}({\bf x}/\epsilon), \; {\bf x} \in {\mathbb
R}^{3},
$$ 
and, for 
$f \in L^{p}(\Omega), 1 \leq p < \infty$, 
$$
f_{\epsilon} = \phi_{\epsilon} \ast f.
$$
\end{definition}
We recall \cite{LiebLoss} that 
$\lim_{\epsilon \rightarrow 0}f_{\epsilon} = f$ 
in 
$L^{p}$ 
and 
$\|f_{\epsilon}\|_{L_{p}} \leq  \|f\|_{L_{p}}, \; \forall \epsilon > 0$.
\begin{definition}
\label{smoothpotdef}
We denote by $\Phi_{\epsilon}$ a smoothed replacement of $\Phi_{\mbox{\rm
qc}}$ as follows.
\begin{enumerate}
\item
$\Phi_{\mbox{\rm lda}} \mapsto \phi_{\epsilon} \ast
\Phi_{\mbox{\rm lda}}$. 
\item
$\Phi_{\mbox{\rm c}} \mapsto \phi_{\epsilon} \ast
\Phi_{\mbox{\rm c}}$. 
\item
Time-history terms are not smoothed.
\end{enumerate}
The effective potential for the approximate problem is given by:
\begin{equation}
\label{smootheff}
V_{\rm e}({\bf x}, t, \rho_{\epsilon}) =
 V({\bf x}, t) + 
W \ast \rho_{\epsilon} + \Phi_{\epsilon}({\bf x}, t, \rho_{\epsilon}).
\end{equation}
\end{definition}
\subsection{Existence and uniqueness for the smoothed system}
As mentioned in the introduction, we will show that the smoothed problem
has a unique weak solution on $[0, T]$ for each fixed $\epsilon > 0$.  
We first state the result.
\begin{proposition}
\label{2.1}
If 
$\Phi_{\rm qc}$ 
is replaced by its smoothing 
$\Phi_{\epsilon}$, as specified in Definition \ref{smoothpotdef},  
then the hypotheses of section \ref{hyps} hold, as applied to
$\Phi_{\epsilon}$.
In particular, Theorem \ref{EU} is applicable.
With $V_{\rm e}$ defined by (\ref{smootheff}), 
there exists a unique weak solution 
$\Psi_{\epsilon}$,
as specified in Definition 1.1, 
of the corresponding system:
\begin{equation}
i \hbar\langle \frac{\partial\Psi_{\epsilon}(t)}{\partial t},
\zeta \rangle  = 
\int_{\Omega} \frac{{\hbar}^{2}}{2m}
\nabla \Psi_{\epsilon}({\bf x}, t)\cdotp \nabla { \zeta}({\bf x}) 
+ V_{\rm e}({\bf x},t,\rho_{\epsilon}) 
\Psi_{\epsilon}({\bf x},t) { \zeta}({\bf x})\; d{\bf x}.  
\end{equation}
\end{proposition}
\begin{proof}
We observe that the time-history term, if present, is assumed to satisfy
the assumptions of section \ref{hyps}. This includes 
(\ref{constraint}) and (\ref{sufsmall}), 
which are required to hold
in the aggregate, 
inclusive of all nonpositive terms for the potential $\Phi_{\epsilon}$. 
The Coulomb potential does not depend on $t$ or
$\rho$; although the unsmoothed potential fails to be in $W^{1,3}$, 
its smoothing is in this space. 
Since individual terms of 
$\phi_{\epsilon} \ast \Phi_{\rm c}$ may be negatively signed, we
estimate the collective potential. We show that this potential satisfies 
(\ref{constraint}) and
(\ref{sufsmall}), with $C_{1}$ preselected to be arbitrarily small.
Initially, we estimate, for $\eta> 0$ arbitrary, 
\begin{equation}
\label{Coulomb1}
\|(\phi_{\epsilon} \ast \Phi_{\rm c}) |\Psi|^{2} \|_{L^{1}} \leq 
(1/2)[\eta^{2} \|(\phi_{\epsilon} \ast \Phi_{\rm c}) \Psi\|_{L^{2}}^{2} 
+ \eta^{-2} 
\|\Psi\|_{L^{2}}^{2}].
\end{equation}
By the H\"{o}lder inequality, with conjugate indices $p=3, 
p^{\prime} = 3/2$, we have
\begin{equation}
\label{Coulomb2}
\|(\phi_{\epsilon} \ast \Phi_{\rm c}) \Psi\|_{L^{2}}^{2} \leq 
[\|\phi_{\epsilon} \ast \Phi_{\rm c}\|_{L^{3}}  
\|\Psi\|_{L^{6}}]^{2} \leq 
[\|\phi_{1} \|_{L^{3}} \|\Phi_{\rm c}\|_{L^{1}}
\|\Psi\|_{L^{6}}]^{2}.  
\end{equation}
By the equivalence of norms on $H^{1}_{0}$, and by Sobolev's inequality,
we may select $\eta$ so that (\ref{sufsmall}) holds for any preselected
$C_{1}$. 
This verifies the final requirement for the Coulomb potential.

For the smoothing of $\Phi_{\rm lda}$, 
we state the three properties required to be verified.
\begin{enumerate}
\item
$ \Phi_{\epsilon}$ 
maps sets bounded in 
$H^{1}_{0}$ 
into sets
bounded in 
$W^{1,3}$.
\item
The Lipschitz property (\ref{ecfollowsH}) holds.
\item
If $\lambda < 0$, 
$\|\phi_{\epsilon} \ast \Phi_{\rm lda} (\rho) |\Psi|^{2} \|_{L^{1}} \leq 
C_{1} \|\nabla \Psi \|_{H^{1}_{0}}^{2}$,
where $C_{1}$ does not depend on $t$ and satisfies (\ref{sufsmall}).
This is a case where $C_{2} = 0$.
\end{enumerate}
Before verifying properties (1) and (2), we note that there is no
restriction on the size of 
$|\lambda|$, 
and the range of 
$\alpha$ 
is
$1 \leq \alpha < 4$, 
whatever the sign of 
$\lambda$.

Property (1) is immediate from the inequalities,
$$
\|\phi_{\epsilon} \ast \Phi_{\rm lda}(\rho)\|_{L^{3}} 
\leq |\lambda| \; \|\phi_{\epsilon}\|_{L^{3}} \||\Psi|^{\alpha}\|_{L^{1}},
\;\;
\|\nabla \phi_{\epsilon} \ast \Phi_{\rm lda}(\rho)\|_{L^{3}} 
\leq |\lambda|\; \|\nabla \phi_{\epsilon}\|_{L^{3}} \||\Psi|^{\alpha}\|_{L^{1}},
$$
which follow from Young's inequality, applied to the convolution.
Indeed, recall that 
$\alpha < 4$, 
so
that the Sobolev inequality may be applied.

For the verification of property (2), we  begin with the gradient term,
and specifically with the product rule 
as applied to the definition of 
$\phi_{\epsilon} \ast \Phi_{\rm lda}/|\lambda|$:
$$
\|\nabla [(\phi_{\epsilon}\ast |\Psi_{1}|^{\alpha}-
\phi_{\epsilon}\ast |\Psi_{2}|^{\alpha}) \psi]\|_{L^{2}}
=
$$
\begin{equation}
\label{te}
\|\nabla \phi_{\epsilon}\ast (|\Psi_{1}|^{\alpha}-
|\Psi_{2}|^{\alpha}) \psi\ + 
\phi_{\epsilon}\ast (|\Psi_{1}|^{\alpha}-
|\Psi_{2}|^{\alpha}) \nabla \psi
\|_{L^{2}}.
\end{equation} 
We have used the differentiation property of the convolution. 
When the triangle inequality is employed, the second term is the more
delicate to estimate since 
$\nabla \psi \in L^{2}$ 
(only). 
Thus, by use of the
Schwarz inequality and Young's inequality, we must estimate
$
\||\Psi_{1}|^{\alpha}-
|\Psi_{2}|^{\alpha} \|_{L^{1}}.
$ 
The case 
$\alpha = 1$ 
is immediate.
We prepare for the cases 
$1 < \alpha < 4$ 
by citing the following useful numerical inequality
\cite{LS}:
\begin{equation}
\label{leach}
\left(\frac{y^{r} - z^{r}}{y^{s} - z^{s}} \frac{s}{r} \right)
^{\frac{1}{r-s}} \leq \max(y,z), 
\; y \geq 0, z \geq 0, y \not=z, r>0, s>0, s \not=r.
\end{equation}
We apply (\ref{leach}) with the identifications.
$$
r = \alpha, s = 1, y = |\Psi_{1}|, z = |\Psi_{2}|, 
$$
to obtain the pointwise estimate, which holds almost
everywhere in $\Omega$,
\begin{equation}
\label{estPhi}
|\;|\Psi_{1}|^{\alpha}-|\Psi_{2}|^{\alpha}|
\leq \alpha (\max(|\Psi_{1}|, |\Psi_{2}|))^{\alpha - 1}
\;|\;|\Psi_{1}| - |\Psi_{2}|\;|. 
\end{equation}
Although we will require inequality (\ref{estPhi}) later in the article,
it is more convenient here to use the less sharp inequality, derived from
(\ref{estPhi}):
$$
|\;|\Psi_{1}|^{\alpha}-|\Psi_{2}|^{\alpha}|
\leq \alpha (1 + |\Psi_{1}|+ |\Psi_{2}|)^{\alpha}
\;|\;|\Psi_{1}| - |\Psi_{2}|\;|. 
$$
We use a technique motivated by \cite{Caz}. If 
$r = \alpha + 2$,
and 
$r^{\prime}$ 
is conjugate to 
$r$, 
if 
$p = r/r^{\prime}$, 
and
$p^{\prime}$ 
is conjugate to
$p$, 
then
\begin{equation}
\label{successiveindices}
\alpha r^{\prime} p^{\prime} = r, \; r^{\prime} p = r,
\end{equation}
and an application of H\"{o}lder's inequality gives
$$
\|\;|\Psi_{1}|^{\alpha}-|\Psi_{2}|^{\alpha}\|_{L^{r^{\prime}}}
\leq \alpha \|1 + |\Psi_{1}|+ |\Psi_{2}|\|_{L^{r}}^{\alpha}
\|\;|\Psi_{1}| - |\Psi_{2}|\;\|_{L^{r}} 
\leq C\|\Psi_{1} - \Psi_{2}\|_{L^{r}}. 
$$ 
An application of Sobolev's inequality shows that the rhs of this
inequality 
is dominated by a locally bounded constant times 
$\|\Psi_{1} - \Psi_{2}\|_{H^{1}}$. 
Since the 
$L^{1}$ 
norm is dominated by a constant times 
the 
$L^{r^{\prime}}$ 
norm, the estimation of the second term 
arising from (\ref{te}) is completed. The first term 
also reduces to the estimation of
$
\||\Psi_{1}|^{\alpha}-
|\Psi_{2}|^{\alpha} \|_{L^{1}},
$ 
as does the non-gradient term.
Thus, the proof of property 
(2) is completed. 

For property (3), which corresponds to $\lambda < 0$ and $1
\leq \alpha \leq 4/3$,  
we consider the following estimate
via two applications of H\"{o}lder's inequality:
\begin{equation}
\label{proplamb}
|\lambda|\left| \int_{\Omega} 
|\phi_{\epsilon} \ast \Psi_{\epsilon}|^{\alpha} |\Psi_{\epsilon}|^{2} 
\; d{\bf x} \right|  
\leq |\lambda| \; |\Omega|^{2/3 - \alpha/2}  
\|\Psi_{\epsilon}\|_{L^{2}}^{\alpha} \; \|\Psi_{\epsilon} \|_{L^{6}}^{2}. 
\end{equation}
Since the $L^{2}$ norm of $\Psi = \Psi_{\epsilon}$ is specified in 
(\ref{constraint}), 
$\lambda$ can be chosen to 
satisfy (\ref{sufsmall}) by use of the Sobolev embedding theorem. 
It follows that a unique weak solution
$\Psi_{\epsilon}$ 
exists for the smoothed system as formulated. 
\end{proof}
\section{Existence}
\setcounter{remark}{1}
The results of this section are derived for an arbitrary time interval
$[0,T]$. They are directed toward the existence statement in Theorem
\ref{central}.
The compactness techniques are motivated by \cite{Caz}.

\subsection{`A priori' bounds for the smoothed solutions}
\label{cofe}
We begin by quoting a result proved in \cite{J1}, now applied to the 
family of solutions 
$\Psi_{\epsilon}$. {\color{blue}{We have absorbed constants into the conserved
Hamiltonian quantities associated with $\Phi_{\epsilon}$.
Thus, $\lambda \mapsto \frac{2 \lambda}{\alpha
+2}$ 
with a similar statement for the leading term of $\Phi$.}} 
\begin{lemma}
\label{lemma3.1}
If the functional 
${\mathcal E}(t)$ 
is defined 
for 
$0 < t \leq T$ 
by,  
\begin{equation}
\label{Eoft}
{\mathcal E}(t) =
\int_{\Omega}\left[\frac{{\hbar}^{2}}{4m}|\nabla \Psi_{\epsilon}|^{2} 
+ 
\left(\frac{1}{4}(W \ast |\Psi_{\epsilon}|^{2})+ \frac{1}{2} 
(V+\Phi_{\epsilon}(\cdotp, t,\rho_{\epsilon}))\right)
|\Psi_{\epsilon}|^{2}\right]d{\bf x},
\end{equation}
then the following identity holds:
\begin{equation}
{\mathcal E}(t)={\mathcal E}(0)
+
\frac{1}{2}\int_{0}^{t}\int_{\Omega}[(\partial V/\partial s)({\bf x},s)
+ \phi({\bf x}, s)]
|\Psi_{\epsilon}|^{2}\;d{\bf x}ds,
\label{consener}
\end{equation}
where 
${\mathcal E}(0)$
is given by
$$ 
\int_{\Omega}\left[\frac{{\hbar}^{2}}{4m}|\nabla
\Psi_{0}|^{2}+\left(\frac{1}{4} 
(W\ast|\Psi_{0}|^{2})+\frac{1}{2}
(V(\cdotp,0) + \Phi_{\epsilon}(\cdotp, 0, \rho_{0})
\right)|\Psi_{0}|^{2}\right]
\;d{\bf x}.
$$
\end{lemma}
\begin{proposition}
\label{3.1}
The kinetic term is bounded above by a natural splitting.
For each fixed $t$:
\begin{equation*}
\frac{{\hbar}^{2}}{4m}\int_{\Omega} |\nabla
\Psi_{\epsilon}|^{2}\;d{\bf x} \leq {\mathcal F}_{\epsilon}(t) +
{\mathcal G}_{\epsilon}(t). 
\end{equation*} 
Here, ${\mathcal F}_{\epsilon}(t)$ is a quantity which can be bounded
above, independently of $t$ and $\epsilon$, in a manner depending only on
the data of the problem. It is given explicitly by
\begin{equation*}
{\mathcal F}_{\epsilon}(t) = {\mathcal E}(0) + \frac{1}{2} 
\int_{0}^{t} \int_{\Omega}
[(\partial V/\partial s)({\bf x}, s) + \phi({\bf x},
s)]|\Psi_{\epsilon}|^{2} d{\bf x} ds - \frac{1}{2} \int_{\Omega}  
V({\bf x}, t) |\Psi_{\epsilon}|^{2}d{\bf x}.
\end{equation*}
Moreover, ${\mathcal G}_{\epsilon}(t)$ can be estimated as the sum of two
terms: the first can be absorbed into the kinetic term, while the second 
is independent of $\epsilon$ and $t$. ${\mathcal G}_{\epsilon}(t)$
is given explicitly by
\begin{equation*}
{\mathcal G}_{\epsilon}(t) =
- \frac{1}{2}\int_{\Omega} 
\Phi_{\epsilon}(\rho_{\epsilon}) |\Psi_{\epsilon}|^{2}d{\bf x}.
\end{equation*}
\end{proposition}
\begin{proof}
\begin{itemize}
\item
The estimation of ${\mathcal F}_{\epsilon}(t)$
\end{itemize}
We notice that $V, \partial V/ \partial t, \phi$ are bounded on the finite
measure space-time domain $\Omega \times [0, T]$, so that the estimation 
of ${\mathcal F}_
{\epsilon}(t)$ reduces to the analysis of the smoothed term 
in ${\mathcal E}(0)$ given by 
\begin{equation*}
\int_{\Omega} \Phi_{\epsilon}(\cdotp, 0, \rho_{0})|\Psi_{0}|^{2} \; d {\bf x}.
\end{equation*}
Since the time-history, if present, is not smoothed, and acts boundedly, it
suffices to examine the Coulomb and LDA potentials. 
\begin{itemize}
\item
The Coulomb term.
\end{itemize}
By the Schwarz inequality and Young's inequality, we estimate
\begin{equation*}
\|(\phi_{\epsilon} \ast \Phi_{\rm c}) |\Psi_{0}|^{2}\|_{L^{1}} \leq 
\|\phi_{1} \|_{L^{2}} \|\Phi_{\rm c}\|_{L^{1}}
\|\Psi_{0}\|_{L^{4}}^{2}.  
\end{equation*}
An application of Sobolev's inequality concludes the argument.
\begin{itemize}
\item
The LDA term.
\end{itemize}
This is a direct estimate:
\begin{equation*}
\|\phi_{\epsilon} \ast \Phi_{\rm lda}(\rho_{0})|\Psi_{0}|^{2}\|_{L^{1}}
\leq \|\phi_{\epsilon} \ast |\Psi|^{\alpha} \|_{L^{3/2}}
\|\Psi_{0} \|_{L^{6}}^{2} \leq \|\phi_{1} \|_{L^{3/2}}
\||\Psi_{0}|^{\alpha} \|_{L^{1}}   
\|\Psi_{0} \|_{L^{6}}^{2}. 
\end{equation*}
Since $\alpha < 4$, the estimate follows as previously from the embedding
theorems.
\begin{itemize}
\item
The estimation of ${\mathcal G}_{\epsilon}(t)$
\end{itemize}
This represents the more delicate part of the proof.  
\begin{itemize}
\item
The time-history term.
\end{itemize}
If the term, 
$$
\Phi({\bf x}, t, \rho)= \Phi({\bf x}, 0, \rho) + 
\int_{0}^{t} \phi({\bf x}, s, \rho) \; ds,
$$  
is included, and the leading term fails to be a positive functional, then
we have required that (\ref{constraint}, \ref{sufsmall}) hold, here as
applied to $\Psi_{\epsilon}$. This is consistent with the structure of 
${\mathcal G}_{\epsilon}$ as stated. The integral term has been discussed
in the previous part and is bounded. Note that (\ref{sufsmall}) is
required to hold for the {\it aggregate} potential, including those
components to be discussed now. We shall mention this at the appropriate
time. 
\begin{itemize}
\item
The Coulomb term.
\end{itemize}
We use the core of the argument as developed in the proof of Proposition 
\ref{2.1}. Indeed, for any preselected $C_{1}$, inequality
(\ref{sufsmall}) can be satisfied. This follows directly from 
(\ref{Coulomb1}) and (\ref{Coulomb2} with a proper choice of $\eta$.
\begin{itemize}
\item
The LDA term.
\end{itemize}
This pertains to the case $\lambda < 0$ if this term is included.
We have already derived the relevant inequality, viz.\thinspace, 
(\ref{proplamb}) near the conclusion of the proof of Proposition \ref{2.1}. 
This inequality is required here also. 

In order to satisfy (\ref{sufsmall}) in the aggregate sense, we reason as
follows. We accept the time-history term as given, if at all. We choose
$\lambda$ so that the sum of the LDA potential and time-history potential
continues to satisfy this inequality. This can be extended to a finite
number of such terms. 
Finally, we have shown that the
Coulomb potential can be included so as to maintain this inequality.
This concludes the proof.
\end{proof}
The following corollary is immediate from the equivalence of norms on
$H^{1}_{0}$.
\begin{corollary}
\label{Hbound}
There is a bound 
$r_{0}$ 
in the norm of 
$C(J; H^{1}_{0})$ 
for the smoothed
solutions. 
\end{corollary}
\begin{proposition}
\label{3.2}
There is a uniform bound, in 
$t \in J$ 
and 
$\epsilon > 0$, 
for the norms,
$$
\|(\Psi_{\epsilon})_{t}\|_{H^{-1}}. 
$$
\end{proposition}
\begin{proof}
One begins by using the weak form of the equation as discussed in
Proposition \ref{2.1},  
and isolating the time
derivative acting on an arbitrary test function $\zeta, \|\zeta
\|_{H^{1}_{0}} \leq 1$.  
The gradient term 
is bounded 
by Corollary \ref{Hbound}, while the bound for the 
external potential term follows directly from the hypothesis on $V$. 
For the Hartree term, we estimate,
by H\"{o}lder's inequality and Young's inequality, for each $t \in J$,
$$
\left|\int_{\Omega} W \ast |\Psi_{\epsilon}|^{2} \; \Psi_{\epsilon}  \zeta
\right|
\leq \|W\|_{L^{1}} \;  
\|\Psi_{\epsilon}\|_{L^{3}}^{2} \|\Psi_{\epsilon}\|_{L^{6}} 
\; \|\zeta\|_{L^{6}}. 
$$ 
Sobolev's inequality, combined with Proposition \ref{3.1}, gives the bound
for this term. 

We now consider the components of the quantum correction
potential.
\begin{itemize}
\item
The LDA term.
\end{itemize}
For the
smoothed LDA term, the sign of $\lambda$ is not relevant
and we consider 
$1 \leq \alpha < 4$. 
We estimate by H\"{o}lder's inequality,
for 
$r=\alpha +2$ 
and
$r^{\prime}$ 
conjugate to 
$r$, for each $t \in J$, 
$$
\left|\int_{\Omega} \phi_{\epsilon} \ast |\Psi_{\epsilon}|^{\alpha}
\; \Psi_{\epsilon}  \zeta
\right| \leq
\|\phi_{\epsilon}\ast|\Psi_{\epsilon}|^{\alpha}
\; \Psi_{\epsilon}\|_{L^{r^{\prime}}} 
\|\zeta\|_{L^{r}}. 
$$
The first factor on the rhs requires additional
explanation. We have, by another application of H\"{o}lder's inequality,
with 
$p = r/r^{\prime}$ 
and 
$p^{\prime}$ 
conjugate to 
$p$ 
(note that 
$r/\alpha = r^{\prime} p^{\prime}$), 
$$
\|\phi_{\epsilon}\ast|\Psi_{\epsilon}|^{\alpha}
\; \Psi_{\epsilon}\|_{L^{r^{\prime}}} 
\leq
\|\phi_{\epsilon} \ast |\Psi_{\epsilon}|^{\alpha}\|_{L^{r/\alpha}}  
\|\Psi_{\epsilon}\|_{L^{r}}
\leq 
\||\Psi_{\epsilon}|^{\alpha}\|_{L^{r/\alpha}}  
\|\Psi_{\epsilon}\|_{L^{r}}
$$
\begin{equation}
\label{sucHolder}
\leq
\|\Psi_{\epsilon}\|_{L^{r}}^{\alpha + 1}.
\end{equation}
We conclude
that the LDA term is bounded in the dual norm, as claimed.
\begin{itemize}
\item
The Coulomb term.
\end{itemize}
By the Schwarz inequality and Young's inequality, uniformly in $t$, 
\begin{equation*}
\left|\int_{\Omega} \phi_{\epsilon} \ast \Phi_{\rm c} 
\; \Psi_{\epsilon}  \zeta
\right| \leq
\|\phi_{1} \|_{L^{2}} \|\Phi_{\rm c}\|_{L^{1}}
\|\Psi_{\epsilon}\|_{L^{4}} \;  
\|\zeta\|_{L^{4}},   
\end{equation*}
and the estimate is completed by Sobolev's inequality.
\begin{itemize}
\item
Time-history term.
\end{itemize}
By Proposition \ref{3.1}, the smoothed solutions are bounded in
$H^{1}_{0}$, uniformly in $t$, so that, by the first hypothesis in section
\ref{hyps}, the functions $\Phi(\cdotp, t, \Psi_{\epsilon})$ have a uniform 
$H^{1}_{0}$ bound. It follows as in previous estimates that the term,
\begin{equation*}
\int_{\Omega} \Phi(\cdotp, 0, \Psi_{\epsilon}) \; \Psi_{\epsilon} \zeta
\; d{\bf x},
\end{equation*} 
defines a functional which
is bounded in the dual norm.
\end{proof}
The following corollary is an immediate consequence of 
Corollary \ref{Hbound} and Proposition
\ref{3.2}.
\begin{corollary}
Any sequence taken from the set 
$\{\Psi_{\epsilon}\}$ 
of 
solutions of the smoothed systems 
is bounded in the norms of 
$C(J; H^{1}_{0})$ 
and  
$C^{1}(J; H^{-1})$.
\end{corollary}
\subsection{Convergent subsequences}
\label{Convsub}
We begin by stating the two  basic lemmas 
derived from the propositions in 
Appendix B.
These are due, in the form stated there, to the authors of 
\cite{Caz} and \cite{Simon}, resp.
\begin{lemma}
\label{l3.2}
There is an element  
$\Psi \in L^{\infty}(J; H^{1}_{0}(\Omega)) \cap W^{1, \infty}(J;
H^{-1}(\Omega))$, 
and a sequence 
$\Psi_{\epsilon_{n}}$ 
satisfying the weak
convergence property,
\begin{equation}
\label{weakh1allt}
\Psi_{\epsilon_{n}}(t) \rightharpoonup \Psi(t), \; \mbox{in} \;
H^{1}_{0}, \;  \forall t \in J. 
\end{equation}
\end{lemma}
\begin{proof}
The preceding corollary, coupled with Proposition \ref{B1}, part (1), 
 furnishes the necessary argument. 
\end{proof}
\begin{lemma}
\label{l3.3}
Suppose $r < 6$ is fixed.
A subsequence of the sequence in 
(\ref{weakh1allt})
may be assumed to converge in $C(J;
L^{r}(\Omega))$.
\end{lemma}
\begin{proof}
The equicontinuity of the sequence from $J$ to $H^{1}_{0}$  
is derived from the fundamental theorem of calculus
applied on an arbitrary subinterval, together with the boundedness
estimates in the dual space.
The compact embedding of $H^{1}_{0} \mapsto L^{r}$, coupled with
Proposition \ref{B2}, furnishes the necessary remaining details. 
We have identified $Y$ with $L^{r}$ here. 
\end{proof}
We divide the verification of Theorem \ref{central} into two parts. 
\begin{theorem}
\label{central1}
The function 
$\Psi$ 
of Lemma \ref{l3.2} 
satisfies the TDDFT system discussed in Theorem \ref{central} with the
quantum corrections.
\end{theorem}
\begin{proof}
By Lemma \ref{l3.3}, by relabelling if necessary, it follows that
\begin{equation}
\label{strongrallt}
\Psi_{\epsilon_{n}}(t) \rightarrow  \Psi(t), \; \mbox{in} \;
L^{r}, \; \mbox{uniformly} \; \forall t \in J, 
\end{equation}
for an arbitrary $r< 6$ selected in advance. 
It follows that 
$\Psi \in C(J; L^{r})$. 
We now examine the equation satisfied by 
$\Psi$. 
By weak convergence (Lemma \ref{l3.2}),
\begin{equation}
\label{lim1}
\lim_{n \rightarrow \infty}
\int_{\Omega} \frac{{\hbar}^{2}}{2m}
\nabla \Psi_{\epsilon_{n}}({\bf x}, t)\cdotp \nabla \zeta({\bf x}) 
\; d{\bf x} 
= \int_{\Omega} \frac{{\hbar}^{2}}{2m}
\nabla \Psi({\bf x}, t)\cdotp \nabla  \zeta({\bf x}) 
\; d{\bf x}. 
\end{equation}
We now consider each of the three cases required to verify that 
\begin{equation}
\label{lim2}
\lim_{n \rightarrow \infty}
\int_{\Omega} 
V_{\rm e}({\bf x},t,\rho_{\epsilon_{n}}) 
\Psi_{\epsilon_{n}}({\bf x},t)  \zeta({\bf x})\; d{\bf x} = 
\int_{\Omega} 
V_{\rm e}({\bf x},t,\rho) 
\Psi({\bf x},t)  \zeta({\bf x})\; d{\bf x}.
\end{equation}
By the boundedness of the external potential, and the strong convergence of
the sequence, we conclude immediately that, for each $t$, 
\begin{equation}
\label{limpot1}
\lim_{n \rightarrow \infty}
\int_{\Omega} 
V({\bf x},t) 
\Psi_{\epsilon_{n}}({\bf x},t)  \zeta({\bf x})\; d{\bf x} = 
\int_{\Omega} 
V({\bf x},t) 
\Psi({\bf x},t) { \zeta}({\bf x})\; d{\bf x}.
\end{equation}
For the Hartree potential, we will use the triangle inequality. Thus, we
begin by writing,
\begin{eqnarray*}
\int_{\Omega} 
W \ast \rho_{\epsilon_{n}} 
\Psi_{\epsilon_{n}}({\bf x},t)  \zeta({\bf x})\; d{\bf x} &-& 
\int_{\Omega} 
W \ast \rho \;
\Psi({\bf x},t)  \zeta({\bf x})\; d{\bf x} = 
 \\
\int_{\Omega} 
W \ast \rho_{\epsilon_{n}} 
[\Psi_{\epsilon_{n}}({\bf x},t)-\Psi({\bf x}, t)] \zeta({\bf x})\; d{\bf x}
&+& \int_{\Omega} 
W \ast [\rho_{\epsilon_{n}} - \rho]
\Psi({\bf x}, t) \zeta({\bf x})\; d{\bf x}.
\end{eqnarray*}
Each of the two rhs terms is estimated by the generalized
H\"{o}lder inequality.
This reduces to estimating the following two triple products of norms:
$$
\|W \ast \rho_{\epsilon_{n}}\|_{L^{2}} 
\|\Psi_{\epsilon_{n}}(t)-\Psi(t)\|_{L^{3}}  \|\zeta\|_{L^{6}}, \; \;
\|W \ast [\rho_{\epsilon_{n}} - \rho]\|_{L^{2}}
\|\Psi(t)\|_{L^{3}} 
\|\zeta\|_{L^{6}}.
$$
For the first triple product, Young's inequality is applied to the convolution
term, followed by $L^{2}$ boundedness; 
$L^{3}$ convergence is applied to the second term of the first product; 
and Sobolev's inequality is applied to the third term.
For the second triple product, the only term requiring explanation is the 
convolution term of the product. We estimate as follows.
$$
\|W \ast [\rho_{\epsilon_{n}} - \rho]\|_{L^{2}}
\leq \|W\|_{L^{2}} \|(|\Psi_{\epsilon_{n}}| - |\Psi|)
(|\Psi_{\epsilon_{n}}| + |\Psi|)\|_{L^{1}},
$$
which is estimated by the Schwarz inequality. An application of $L^{2}$
boundedness and $L^{2}$ convergence yields the final result:
\begin{equation}
\label{limpot2}
\lim_{n \rightarrow \infty}
\int_{\Omega} 
W \ast \rho_{\epsilon_{n}} 
\Psi_{\epsilon_{n}}({\bf x},t)  \zeta({\bf x})\; d{\bf x} = 
\int_{\Omega} 
W \ast \rho \;
\Psi({\bf x},t)  \zeta({\bf x})\; d{\bf x}.
\end{equation}
The potential $\Phi_{\rm qc}$ requires the analysis of the three 
components introduced in section \ref{qcsection}. 
For the smoothed LDA potential $\phi_{\epsilon} \ast \Phi_{\rm lda}$, 
we will use the triangle inequality, and
we write,
\begin{equation*}
\int_{\Omega} 
\phi_{\epsilon_{n}} \ast \Phi_{\rm lda}(\rho_{\epsilon_{n}})  
\Psi_{\epsilon_{n}}({\bf x},t)  \zeta({\bf x})\; d{\bf x} - 
\int_{\Omega} 
\Phi_{\rm lda}(\rho)
\Psi({\bf x},t)  \zeta({\bf x})\; d{\bf x} = 
\end{equation*}
\begin{equation*} 
\int_{\Omega} 
\phi_{\epsilon_{n}} \ast \Phi_{\rm lda}(\rho_{\epsilon_{n}})  
[\Psi_{\epsilon_{n}}({\bf x},t)-\Psi({\bf x}, t)] \zeta({\bf x})\; d{\bf x}
+ 
\end{equation*}
\begin{equation*} 
\int_{\Omega} 
[\phi_{\epsilon_{n}} \ast \Phi_{\rm lda}(\rho_{\epsilon_{n}})  -\Phi_{\rm lda}(\rho)]
\Psi({\bf x}, t) \zeta({\bf x})\; d{\bf x}.
\end{equation*}
We apply the H\"{o}lder inequality to each of the 
terms to 
obtain two products
of norms:
$$
\|\phi_{\epsilon_{n}} \ast \Phi_{\rm lda}(\rho_{\epsilon_{n}})  
[\Psi_{\epsilon_{n}}(t)-\Psi(t)]\|_{L^{r^{\prime}}} 
\|\zeta\|_{L^{r}}, \; \; 
\|[\phi_{\epsilon_{n}} \ast \Phi_{\rm lda}(\rho_{\epsilon_{n}})  -\Phi_{\rm lda}(\rho)]
\Psi(t)\|_{L^{r^{\prime}}} \| \zeta \|_{L^{r}},
$$
where $r = \alpha + 2$ and $r^{\prime}$ is conjugate to $r$.
We use the method employed in the proof of Proposition
\ref{3.2} (cf.\thinspace (\ref{sucHolder})) 
in order
to estimate the $L^{r^{\prime}}$ norms. 
For convenience, we suppress the scalar $|\lambda|$; also, $1 \leq \alpha
< 4$.
We have, for the first product,
$$
\|\phi_{\epsilon_{n}} \ast \Phi_{\rm lda}(\rho_{\epsilon_{n}})  
[\Psi_{\epsilon_{n}}(t)-\Psi(t)]\|_{L^{r^{\prime}}} \leq 
\|\phi_{\epsilon_{n}} \ast |\Psi_{\epsilon_{n}}|^{\alpha}\|_{L^{r/\alpha}}  
\|\Psi_{\epsilon_{n}}(t)-\Psi(t)]\|_{L^{r}} 
\leq 
$$
$$
\||\Psi_{\epsilon_{n}}|^{\alpha}\|_{L^{r/\alpha}}  
\|\Psi_{\epsilon_{n}}(t)-\Psi(t)]\|_{L^{r}} 
\leq 
\|\Psi_{\epsilon_{n}}\|_{L^{r}}^{\alpha}  
\|\Psi_{\epsilon_{n}}(t)-\Psi(t)]\|_{L^{r}}, 
$$
which converges to zero as remarked at the beginning of the proof
(see (\ref{strongrallt})).
Thus, the first product of norms is convergent to zero. For the second
product, we begin as before, to obtain,
$$
\|[\phi_{\epsilon_{n}} \ast \Phi_{\rm lda}(\rho_{\epsilon_{n}})  -\Phi_{\rm lda}(\rho)]
\Psi(t)\|_{L^{r^{\prime}}} \leq
\|\phi_{\epsilon_{n}} \ast \Phi_{\rm lda}(\rho_{\epsilon_{n}})  -\Phi_{\rm lda}(\rho)\|
_{L^{r/\alpha}} \|\Psi(t)\|_{L^{r}}.
$$
To estimate this, we apply the triangle inequality to the first factor:
$$
\|\phi_{\epsilon_{n}} \ast \Phi_{\rm lda}(\rho_{\epsilon_{n}})  -\Phi_{\rm lda}(\rho)\|
_{L^{r/\alpha}}  \leq
\|\phi_{\epsilon_{n}} \ast \Phi_{\rm lda}(\rho_{\epsilon_{n}}) - 
\phi_{\epsilon_{n}} \ast \Phi_{\rm lda}(\rho)\| 
_{L^{r/\alpha}} + 
$$
$$
\|\phi_{\epsilon_{n}} \ast \Phi_{\rm lda}(\rho) - \Phi_{\rm lda}(\rho)\|
_{L^{r/\alpha}}.
$$
The first term on the rhs is bounded, via the smoothing property, by 
$$
\|\phi_{\epsilon_{n}} \ast \Phi_{\rm lda}(\rho_{\epsilon_{n}}) - 
\phi_{\epsilon_{n}} \ast \Phi_{\rm lda}(\rho)\| 
_{L^{r/\alpha}} \leq 
\||\Psi_{\epsilon_{n}}|^{\alpha} - 
|\Psi|^{\alpha}\| 
_{L^{r/\alpha}}.
$$
The estimation of this expression 
requires inequality (\ref{estPhi}) with the identifications 
$\Psi_{1} \mapsto \Psi_{\epsilon_{n}}, \Psi_{2} \mapsto \Psi$. 
When the power $r/\alpha$ is applied to the inequality, 
and integration over $\Omega$ is carried out,
one can apply H\"{o}lder's inequality with $p = \alpha$ and $p^{\prime} = 
\alpha/(\alpha - 1)$ to conclude convergence. Convergence for the second
term is a consequence of the property of smoothing; since
$|\Psi|^{\alpha} \in L^{r/\alpha}$, its convolution is convergent in norm. 
Altogether, we have shown:
\begin{equation}
\label{limpot3}
\lim_{n \rightarrow \infty}
\int_{\Omega} 
\phi_{\epsilon_{n}} \ast \Phi_{\rm lda}
(\rho_{\epsilon_{n}}) 
\Psi_{\epsilon_{n}}({\bf x},t)  \zeta({\bf x})\; d{\bf x} = 
\int_{\Omega} 
\Phi_{\rm lda}(\rho) 
\Psi({\bf x},t)  \zeta({\bf x})\; d{\bf x}.
\end{equation}
We now consider the Coulomb term. Again, we write
\begin{eqnarray*}
\int_{\Omega} 
\phi_{\epsilon_{n}} \ast \Phi_{\rm c}  
\Psi_{\epsilon_{n}}({\bf x},t)  \zeta({\bf x})\; d{\bf x} &-& 
\int_{\Omega} 
\Phi_{\rm c}
\Psi({\bf x},t)  \zeta({\bf x})\; d{\bf x} = 
 \\
\int_{\Omega} 
\phi_{\epsilon_{n}} \ast \Phi_{\rm c}  
[\Psi_{\epsilon_{n}}({\bf x},t)-\Psi({\bf x}, t)] \zeta({\bf x})\; d{\bf x}
&+& 
\int_{\Omega} 
[\phi_{\epsilon_{n}} \ast \Phi_{\rm c}  -\Phi_{\rm c}]
\Psi({\bf x}, t) \zeta({\bf x})\; d{\bf x}.
\end{eqnarray*}
The estimation is now straightforward. The H\"{o}lder inequality yields
the two triple products for the rhs term estimates:
$$
\|\phi_{\epsilon_{n}} \ast \Phi_{\rm c}\|_{L^{2}} \;
\|\Psi_{\epsilon_{n}}(t) - \Psi(t) \|_{L^{3}} \;
\|\zeta\|_{L^{6}}, \;
\|\phi_{\epsilon_{n}} \ast \Phi_{\rm c} - \Phi_{\rm c} \|_{L^{2}} \;
\|\Psi(t)\|_{L^{3}} \; \|\zeta \|_{L^{6}}.
$$
The first term is convergent because of strong convergence; the second,
because of the convergence of the smoothing in $L^{2}$.

The final term to estimate among the quantum correction terms is the
time-history term, if present. Recall that this term is not smoothed. 
The term $\Phi(\cdotp, t, \rho)$
is analyzed as follows. We have the algebraic representation, 
\begin{equation*}
\int_{\Omega} \Phi(\cdotp, t, \rho_{\epsilon_{n}}) \; \Psi_{\epsilon_{n}} \zeta
\; d{\bf x} -
\int_{\Omega} \Phi(\cdotp, t, \rho) \; \Psi \zeta
\; d{\bf x} =
\end{equation*} 
\begin{equation*}
\int_{\Omega}[\Phi(\cdotp, t, \rho_{\epsilon_{n}}) - \Phi(\cdotp, t,
\rho)] \; 
\Psi_{\epsilon_{n}} \zeta
\; d{\bf x} \; + 
\end{equation*}
\begin{equation*}
\int_{\Omega}[\Phi(\cdotp, t, \rho)[\Psi_{\epsilon_{n}} - \Psi)] \zeta
\; d{\bf x}.
\end{equation*}
The first term converges to zero because of the 
assumed uniform $L^{2}$ continuity of $\Phi$ 
in its third argument, 
while
the second term is governed by the uniform convergence in $L^{r}$.

We now use (\ref{lim1}) and (\ref{lim2}) to conclude that
$$
\lim_{n \rightarrow \infty} \langle \partial \Psi_{\epsilon_{n}}/\partial
t, \zeta \rangle
= \int_{\Omega} \frac{{\hbar}^{2}}{2m}
\nabla \Psi({\bf x}, t)\cdotp \nabla  \zeta({\bf x}) 
+ V_{\rm e}({\bf x},t,\rho) 
\Psi({\bf x},t)  \zeta({\bf x})\; d{\bf x}.
$$
However, we may deduce from Lemma \ref{l3.2} that 
\begin{equation}
\label{deducefrom}
\lim_{n \rightarrow \infty} 
\langle \partial \Psi_{\epsilon_{n}}/\partial t, \zeta \rangle
= 
\langle \partial \Psi/\partial t, \zeta \rangle,
\end{equation}
so that $\Psi$ solves the TDDFT system. The initial condition is a
consequence of (\ref{strongrallt}) in, say, $L^{2}$ for $t=0$.
\end{proof}
It remains to verify the regularity class for $\Psi$.
\begin{theorem}
\label{central2}
The function $\Psi$ of Theorem \ref{central1} satisfies
$$\Psi \in C(J; H^{1}_{0}(\Omega)) \cap C^{1}(J;
H^{-1}(\Omega)).$$ 
\end{theorem}
\begin{proof}
We begin with the verification that $\Psi \in C(J; H^{1}_{0})$,
and make use of Proposition \ref{B1}, part (2), of appendix B.
In particular, it suffices to show that 
\begin{equation*}
\int_{\Omega}\frac{{\hbar}^{2}}{4m}|\nabla \Psi_{\epsilon_{n}}|^{2} 
\; d{\bf x}
\rightarrow 
\int_{\Omega}\frac{{\hbar}^{2}}{4m}|\nabla \Psi|^{2}\; d{\bf x}, \; n
\rightarrow \infty, \; \mbox{uniformly in} \; t. 
\end{equation*}
We use the representations contained in Lemma \ref{lemma3.1} as applied to
$\Psi_{\epsilon_{n}}$. We rewrite them as follows.
\begin{equation}
\label{Eoft2}
{\mathcal E}_{n}(t) =
\int_{\Omega}\left[\frac{{\hbar}^{2}}{4m}|\nabla \Psi_{\epsilon_{n}}|^{2} 
+ 
\left(\frac{1}{4}(W \ast |\Psi_{\epsilon_{n}}|^{2})+ \frac{1}{2} 
(V+\Phi_{\epsilon_{n}}(\cdotp, t,\rho_{\epsilon_{n}}))\right)
|\Psi_{\epsilon_{n}}|^{2}\right]d{\bf x},
\end{equation}
\begin{equation}
{\mathcal E}_{n}(t)={\mathcal E}(0)
+
\frac{1}{2}\int_{0}^{t}\int_{\Omega}[(\partial V/\partial s)({\bf x},s)
+ \phi({\bf x}, s)]
|\Psi_{\epsilon_{n}}|^{2}\;d{\bf x}ds.
\label{consener2}
\end{equation}
Note that the expression ${\mathcal E}_{n}(t)$, as defined in
(\ref{Eoft2}), converges uniformly in $t$ to
${\mathcal E}(t)$, 
when the boundedness for 
$\partial V/ \partial t + \phi$ is
applied, due to strong convergence. 
The approach now is to solve for the gradient term in
(\ref{Eoft2}) and deduce its uniform convergence from that of each of the other
terms. Because of the hypotheses made on the external potential and the
time-history terms,  
the terms requiring analysis are the Hartree and 
remaining quantum correction terms.
The techniques are similar to those used earlier. For the Hartree
potential, we have  
\begin{eqnarray*}
\int_{\Omega} 
W \ast \rho_{\epsilon_{n}}(t) \; 
\rho_{\epsilon_{n}}({\bf x},t) \; d{\bf x} &-& 
\int_{\Omega} 
W \ast \rho(t) \; 
\rho({\bf x},s) \; d{\bf x} = 
 \\
\int_{\Omega} 
W \ast \rho_{\epsilon_{n}}(t) 
[\rho_{\epsilon_{n}}({\bf x},t)-\rho({\bf x}, t)]\; d{\bf x}
&+& \int_{\Omega} 
W \ast [\rho_{\epsilon_{n}}(t)  - \rho(t)]
\rho({\bf x}, t) \; d{\bf x}.
\end{eqnarray*}
Each of the two rhs terms is estimated by the Schwarz inequality, so that
we must estimate the following two products of norms:
$$
\|W \ast \rho_{\epsilon_{n}}(t) \|_{L^{2}}
\|\rho_{\epsilon_{n}}(t) - \rho(t)\|_{L^{2}}, \;
\|W \ast [\rho_{\epsilon_{n}}(t)  - \rho(t)]\|_{L^{2}}
\|\rho(t) \|_{L^{2}}.
$$
For the first product, the first term is estimated by Young's inequality,
to obtain a quantity, bounded on $J$.  
We estimate the second factor as  
$$
\|\rho_{\epsilon_{n}}(t) - \rho(t)\|_{L^{2}} \leq 
\||\Psi_{\epsilon_{n}}(t)| - |\Psi(t)| \|_{L^{4}} 
\||\Psi_{\epsilon_{n}}(t)| + |\Psi(t)| \|_{L^{4}}, 
$$
which is convergent to zero as $n \rightarrow \infty$, by 
the strong uniform convergence. 
For the second product, 
an application of Young's inequality and the strong uniform convergence
allows one to conclude that
uniform convergence to zero as $n
\rightarrow \infty$.
Next, we consider the LDA term.
\begin{eqnarray*}
\int_{\Omega} 
\Phi_{\rm lda}(\rho_{\epsilon_{n}}(t)) \rho_{\epsilon_{n}}({\bf x}, t)
\; d{\bf x} &-& 
\int_{\Omega} 
\Phi_{\rm lda}(\rho (t)) \rho({\bf x}, t)
\; d{\bf x} = 
 \\
\int_{\Omega} 
\Phi_{\rm lda}(\rho_{\epsilon_{n}}(t))
[\rho_{\epsilon_{n}}({\bf x},t)-\rho({\bf x}, t)] \; d{\bf x}
&+& \int_{\Omega} 
[\Phi_{\rm lda}(\rho_{\epsilon_{n}}(t)) -\Phi_{\rm lda}(\rho(t))] \rho({\bf x}, t)
\; d{\bf x}.
\end{eqnarray*}
H\'{o}lder's inequality is applied to each of the terms on the rhs, so
that we need to estimate the following norm products:
$$
\||\Psi_{\epsilon_{n}}(t)|^{\alpha} [|\Psi_{\epsilon_{n}}(t)| - |\Psi(t)|] \|_{L^{r^{\prime}}} \;
\||\Psi_{\epsilon_{n}}(t)| + |\Psi(t)|\|_{L^{r}}, 
$$
$$
\|\;[|\Psi_{\epsilon_{n}}(t)|^{\alpha}- |\Psi(t)|^{\alpha}]|\Psi(t)| \; \|_{L^{r^{\prime}}}
\|\Psi(t)\|_{L^{r}},
$$
where $r = \alpha + 2$ and $r^{\prime}$ is conjugate to $r$.
As has been demonstrated previously, the first product is estimated by
$$
\|\Psi_{\epsilon_{n}}(t)\|_{L^{r}}^{\alpha}\; 
\|\Psi_{\epsilon_{n}}(t) - \Psi(t) \|_{L^{r}} 
(\|\Psi_{\epsilon_{n}}(t)\|_{L^{r}} +  
\|\Psi(t)\|_{L^{r}}),   
$$
which converges to zero as $n \rightarrow \infty$. The second product is
estimated, with the help of (\ref{estPhi}) and H\"{o}lder's inequality, as 
\begin{equation}
\label{alphaminusone}
\alpha \|(|\Psi_{\epsilon_{n}}(t)| + |\Psi(t)|)^{\alpha - 1} (|\Psi_{\epsilon_{n}}(t)| - |\Psi(t)|)
\|_{L^{r/\alpha}} 
\|\Psi(t)\|_{L^{r}}^{2},
\end{equation}
and another application of H\"{o}lder's inequality, with $p=\alpha$ and
$p^{\prime}$ conjugate to $\alpha$, gives the bound, 
$$
\alpha \|(|\Psi_{\epsilon_{n}}(t)| + |\Psi(t)|)\|_{L^{r}}^{\alpha - 1} 
\; \||\Psi_{\epsilon_{n}}(t)| - |\Psi(t)|\|_{L^{r}}
\|\Psi(t)\|_{L^{r}}^{2},
$$ 
so that this term also converges to
zero. Finally, the Coulomb term is directly estimated via the strong
convergence; we omit the details.
It follows that $\Psi \in C(J; H^{1}_{0})$. 

In order to conclude
that $\Psi \in C^{1}(J; H^{-1})$, we subtract two copies of the TDDFT
system, one evaluated at $t$, and the other at $s$, and we estimate for an
arbitrary test function $\zeta$. We need to show that this difference
satisfies a zero limit as $t \rightarrow s$, uniformly in
$\|\zeta\|_{H^{1}_{0}} \leq 1$. 
The property just established, $\Psi \in C(J; H^{1}_{0})$, implies this
for the gradient and external potential terms. The remaining terms can be
estimated via a very useful analogy: replace the $n \rightarrow \infty$
limit in the estimates for Theorem \ref{central1} 
by the $t \rightarrow s$ limit, after constructing parallel algebraic
representations. The
convergence of the corresponding dominating terms holds since 
$\Psi \in C(J; H^{1}_{0})$. 
This completes the proof. 
\end{proof}
\begin{remark}
The combination of Theorem \ref{central1} and Theorem \ref{central2} 
gives Theorem \ref{central} as formulated earlier. This is the 
first central
result of the article.
\end{remark}
\section{Uniqueness}
{\color{blue}{
This section is a replacement for the original section}}.
The following theorem will be established in this section by the techniques
associated with evolution operators. We first state the theorem, and then
establish appropriate background, prior to providing the details of the
proof. We note that the analysis presented here excludes from uniqueness
the case(s) $1 < \alpha < 2$ in the representation of the LDA component of
the potential. 
\begin{theorem}
\label{UTH}
Under the assumptions of this article, there is a unique
weak solution of (\ref{wsol}), where $V_{{\rm e}}$ is defined in
(\ref{redefined}). The defining properties of weak solution are described in 
Definition \ref{weaksolution}. The cases $1 < \alpha < 2$ are excluded.
\end{theorem}
Some results allow for $1 < \alpha < 2$. We will be specific when these
cases are excluded.
\subsection{Background}
The evolution operator permits the solution of the linear Cauchy problem,
\begin{eqnarray}
\frac{du}{dt} + A(t) u(t) &=& F(t), \nonumber \\
u(0) = u_{0},
\label{CP}
\end{eqnarray}
on an interval $[0, T]$, with values in a Banach space.
The solution is given by
\begin{equation}
u(t) = U(t,0) u_{0} + \int_{0}^{t} U(t,s) \; F(s) \; ds,
\label{SCP}
\end{equation}
under (strong) assumptions on $u_{0}, F$. In order that (\ref{SCP}) 
hold rigorously, 
the evolution operators $U(t,s)$ are derived for a pair, $(A(t),
X)$ and $(A(t), Y)$, where $Y$ is continuously embedded in $X$ and is a core
subspace of the domains of $A(t)$. The operators are typically generated
on $X$, and shown to be invariant on $Y$ by a commutator relation. 
For this article, $Y = H^{1}_{0}$ and $X = H^{-1}$. This
theory is due to Kato, and is developed in \cite[Chapter 6]{J2}.  
For this article, we may use the results of \cite{J1}, where the desired
properties of the evolution operators 
were derived for Hamiltonian operators $A(t) = {\hat H}(t)$,  
including the kinetic term
plus the external, Hartree, and time-history potentials. The Coulomb and
LDA potentials were not included in that theory. It follows that any
application of these results, intended to derive uniqueness, 
must shift the Coulomb and LDA terms
into the action of $F(t)$. 
The theory asserts \cite[Prop.\ 6.4.1]{J2} that there is a one-to-one
correspondence between the representation (\ref{SCP}) and the unique
solution of 
(\ref{CP}) if $u_{0} \in Y$ and $F \in C(J; X) \cap L^{1}(J;
Y)$. When this hypothesis holds, the unique solution is in $C^{1}(J; X)
\cap C(J; Y)$. However, the characterization of $F$ in the current
situation does not satisfy the required regularity. This accounts for the
following method which we use. Note that we are able to use the linear
theory by defining coefficients of $A(t)$ in terms of the solution itself. 
Further, for $\rho = |\Psi|^{2}$, define
\begin{equation}
F(\Psi) = -[\Phi_{\rm c} + \Phi_{\rm
lda}(\rho)] \Psi, 
\label{defF}
\end{equation}
and $U^{\rho}(t,s)$ to be the evolution operators derived in \cite{J1},
based on a Hamiltonian including 
the kinetic term
plus the external, Hartree, and time-history potentials. As defined, $F$
fails to be in $C(J; H^{1}_{0})$. 
We make use of the following smoothing.
\begin{definition}
\label{Usmooth}
Consider 
the smoothing of section \ref{smoothing}, and define
\begin{equation}
F_{\epsilon}(\Psi) = -[\phi_{\epsilon}\ast\Phi_{\rm c} + 
\phi_{\epsilon}\ast\Phi_{\rm
lda}(\rho)] \Psi, \; 1 \leq \alpha < 4. 
\end{equation}
\end{definition}
This smoothing will be used to prove the following.
\begin{lemma}
\label{lemmaU2}
Suppose that $F(\Psi(s))$ is defined by (\ref{defF}). 
If $\Psi$ satisfies (\ref{wsol}), then  
\begin{equation}
\Psi(t) = U^{\rho}(t,0) \Psi_{0} + 
\int_{0}^{t} U^{\rho}(t,s) \; F(\Psi(s)) \; ds,
\label{SCP1}
\end{equation}
where the integral is interpreted as a member of $C^{1}(J; H^{-1})$,
and $\Psi$ is interpreted as a distribution. 
Conversely, if $\Psi \in C(J; H^{1}_{0}) \cap C^{1}(J; H^{-1})$ 
satisfies (\ref{SCP1}),
then $\Psi$ satisfies (\ref{wsol}).
\end{lemma} 
\begin{proof}
Suppose $\Psi$ is a solution of (\ref{wsol}), with the specified
regularity. 
The function 
$F_{\epsilon}(\Psi)$
is a member of $C(J; H^{1}_{0})$ and 
the replacement in (\ref{SCP1}) of 
$F(\Psi)$ by
$F_{\epsilon}(\Psi)$ yields a solution $\Psi_{\epsilon}$, 
\begin{equation}
\Psi_{\epsilon}(t) 
 = U^{\rho}(t,0) \Psi_{0} + 
\int_{0}^{t} U^{\rho}(t,s) \; F_{\epsilon}(\Psi(s)) \; ds,
\label{SCPapp}
\end{equation}
of the adjusted equation (\ref{wsol}). 
These may be thought of as nearby approximate linear equations, 
indexed by $\epsilon$. We notice the important fact for the argument that 
the functions $\Psi_{\epsilon}(t)$ form a bounded family,
independent of $\epsilon$, in $C(J;
H^{1}_{0})$. 
We use the specific properties that the convolution terms in the
definition of $F_{\epsilon}$ are $L^{\infty}$ functions, with bound
independent of $\epsilon$, and possess $L^{3}$ derivatives, with this norm
independent of $\epsilon$.

In the smoothed case, there is a one to one correspondence between the
representation and the adjusted system. 
As $ \epsilon \rightarrow 0$, the representations 
$\Psi_{\epsilon}$ 
converge in $C(J; H^{-1})$ to a
representation 
\begin{equation}
\Psi_{\ast}(t) = U^{\rho}(t,0) \Psi_{0} + 
\int_{0}^{t} U^{\rho}(t,s) \; F(\Psi(s)) \; ds.
\label{SCP11}
\end{equation}
The convergence follows from \cite[Prop.\
7.1.1]{J2} and the argument presented now. 
Since the evolution operators are independent of $\epsilon$, it suffices 
to estimate the norm of 
$$
\|F(\Psi) - F_{\epsilon}(\Psi)\|_{L^{1}(J; H^{-1})}.
$$
The LDA component is estimated 
in $C(J; H^{-1})$ as follows. For $1 \leq \alpha < 4$, 
we have that $|\Psi|^{\alpha} \in C(J; L^{3/2})$.
It follows that, for $\phi \in H^{1}_{0}$, 
$$
\|[F(\Psi) - F_{\epsilon}(\Psi)] \phi\|_{L^{1}} \rightarrow 0,
$$ 
uniformly in $t$. This combines the generalized H\"{o}lder inequality and 
the properties of the smoothing in $L^{3/2}$. 

The duality estimate for the Coulomb potential is carried out by a similar 
estimate,
via the generalized H\"{o}lder inequality. All that is required is the 
known convergence of the smoothing in $L^{3/2}$. 
We conclude that (\ref{SCP11}) holds.

It remains to equate
$\Psi_{\ast}$ with $\Psi$. 
If we examine the adjusted system (\ref{wsol}), corresponding to
(\ref{SCPapp}), we conclude that the family
$\partial \Psi_{\epsilon}\partial t$ is bounded in $C(J; H^{-1})$. 
This is implied by the boundedness, already noted, of the family 
$ F_{\epsilon}(\Psi)$ in $C(J; H^{1}_{0})$. 
This in turn yields the equicontinuity required for the application of
Proposition B.1 of appendix B.
When this convergence result is applied to a subsequence of  
$\Psi_{\epsilon}$, 
we conclude that 
the difference $\Psi -
\Psi_{\ast}$ solves a linear initial value problem, 
upon cancellation of the terms involving $F$, for which zero is the
unique solution.  

For the converse, we begin with $\Psi$ satisfying (\ref{SCP1}), 
and define $F_{\epsilon}$ as before. 
We again argue that $\Psi$ satisfies 
(\ref{wsol}) using the same limit analysis.
\end{proof}
The following is immediate from the lower semicontinuity of the norm with
respect to weak convergence.
\begin{corollary}
For the weakly convergent sequence $\Psi_{\epsilon_{n}}$ of the proof, we have
$$
\|\Psi\|_{C(J;H^{1}_{0})} \leq \liminf_{n \rightarrow \infty}
\|\Psi_{\epsilon_{n}}\|_{C(J;H^{1}_{0})}.
$$
\end{corollary}

\subsection{The approximation arguments: Proof of Theorem \ref{UTH}}
In order to establish uniqueness, we consider two separate equations,
defined by ${\hat H}_{\rho_{1}}(t), F(\Psi_{1}(t))$, and 
${\hat H}_{\rho_{2}}(t), F(\Psi_{2}(t))$, with solutions $\Psi_{1}, \Psi_{2}$,
resp. The evolution operators are denoted by $U^{\rho_{j}}(t,s), j =
1,2$. The representations satisfy the lemma and are understood to be
in $C(J; H^{-1})$, since the functions $F(\Psi_{j})$ are
in this space. Explicitly,
\begin{equation}
F(\Psi_{j}) = -[\Phi_{\rm c} + \Phi_{\rm
lda}(\rho_{j})] \Psi_{j}, \; j = 1,2. 
\end{equation}
\begin{proposition}
Suppose that $\Psi_{1}$ and $\Psi_{2}$ are two distinct solutions of 
(\ref{SCP1}), where $\rho_{j} = |\Psi_{j}|^{2}, j = 1,2$. Suppose the
respective approximations are given by (\ref{SCPapp}), written here as 
\begin{equation}
\Psi^{\epsilon}_{j}(t) 
 = U^{\rho_{j}}(t,0) \Psi_{0} + 
\int_{0}^{t} U^{\rho_{j}}(t,s) \; F_{\epsilon}(\Psi_{j}(s)) \; ds.
\label{SCPappj}
\end{equation}
For $\alpha = 1$, or $2 \leq \alpha < 4$,
there exists a constant $C$, not depending on $\epsilon$, such that
\begin{equation}
\|\Psi_{1}^{\epsilon}(t) - \Psi_{2}^{\epsilon}(t)\|_{H^{1}_{0})} \leq
C \int_{0}^{t} \|\Psi_{1}(s) - \Psi_{2}(s)\|_{H^{1}_{0}} \; ds,
\label{Gestimate1}
\end{equation}
for all $t \in [0, T]$.
\end{proposition}
\begin{proof}
We begin the argument by writing the operator difference,
$$
\Psi_{1}^{\epsilon}(t) - \Psi_{2}^{\epsilon}(t) = 
[U^{\rho_{1}}(t,0) - U^{\rho_{2}}(t,0)] \Psi_{0} + 
\int_{0}^{t}
[U^{\rho_{1}}(t,s)-U^{\rho_{2}}(t,s)]
\;F_{\epsilon}(\Psi_{1}(s))\;ds \;+
$$
\begin{equation}
\int_{0}^{t} U^{\rho_{2}}(t,s)\;
[F_{\epsilon}(\Psi_{1}(s))-F_{\epsilon}(\Psi_{2}(s))]\;ds.
\label{eqdiff1}
\end{equation}
The estimations of the first and second terms depend
on the representations \cite[(7.1.3)]{J2}, for $g \in H^{1}_{0}$,
\begin{equation}
U^{\rho_{1}}(t,r) g - U^{\rho_{2}}(t,r) g = - \int_{r}^{t} U^{\rho_{1}}
(t,s) [{\hat H}^{\rho_{1}}(s) - {\hat H}^{\rho_{2}}(s)]
U^{\rho_{2}}(s,r)g\;ds. 
\label{eqdiff2}
\end{equation}
Here,
\begin{equation}
[{\hat H}^{\rho_{1}}(s) - {\hat H}^{\rho_{2}}(s)]g = 
 W \ast [\rho_{1} - \rho_{2}]g + 
[\Phi(\Psi_{1}) - \Phi(\Psi_{2})]g
\label{eqdiff3}
\end{equation}
is a member of $C(J; H^{1}_{0})$ if $g \in C(J; H^{1}_{0})$.
We have employed cancellation of the kinetic term and the external
potential term in (\ref{eqdiff3}).
Both terms are readily estimated, uniformly in $s$, in the $H^{1}_{0}$
norm, the first by \cite[Theorem 3.1]{JNA}, for $g = \Psi_{0}$,
and the second by the hypothesis assumed for $\Phi$, for $g =  
F_{\epsilon}(\Psi_{1}(s))$.
After the action of the evolution operators, with respect to bounded sets
in $H^{1}_{0}$, uniformly in $t \in J$, is taken into account, we obtain
an estimate of the form (\ref{Gestimate1}) for these terms in
(\ref{eqdiff1}).

For the estimation of the third term in 
(\ref{eqdiff1}), 
we may write the preliminary algebraic step as
\begin{equation}
\phi_{\epsilon} \ast |\Psi_{1}|^{\alpha}\Psi_{1} -
\phi_{\epsilon} \ast |\Psi_{2}|^{\alpha}\Psi_{2} = 
\phi_{\epsilon} \ast [|\Psi_{1}|^{\alpha} - |\Psi_{2}|^{\alpha}]\Psi_{1} +
\phi_{\epsilon} \ast |\Psi_{2}|^{\alpha}[ \Psi_{1} - \Psi_{2}].
\label{algstep}
\end{equation}
The $H^{1}_{0}$ seminorm requires the estimation of four 
rhs terms for this relation as computed
by the product rule for differentiation.
We provide a summary analysis of each term involved in the
differentiation of (\ref{algstep}). We must show that such terms are
Lipschitz in $L^{2}$, uniformly in $t \in J$. 
The partial derivative, 
with respect to $x_{j}$, of the first term has a pointwise upper bound given by
\begin{equation}
\alpha \phi_{\epsilon} \ast  
||\Psi_{1}|^{\alpha -1}\partial \Psi_{1}/\partial x_{j}
- |\Psi_{2}|^{\alpha -1}\partial \Psi_{2}/\partial x_{j}|\;(|\Psi_{1}|)
+ 
\phi_{\epsilon} \ast ||\Psi_{1}|^{\alpha} - |\Psi_{2}|^{\alpha}|
\; (|\partial \Psi_{1}/\partial x_{j}|). 
\label{first}
\end{equation}
The second term 
of (\ref{first}) is readily estimated, since the convolution factor is 
uniformly in $L^{\infty}$ by Young's convolution inequality. 
In fact, 
$$
\|\phi_{\epsilon} \ast ||\Psi_{1}|^{\alpha} -
|\Psi_{2}|^{\alpha}|\|_{L^{\infty}} \leq
\|\phi_{\epsilon}\|_{L^{r}} \; 
\||\Psi_{1}|^{\alpha} -
|\Psi_{2}|^{\alpha}|\|_{L^{r^{\prime}}},
$$
where $r = \alpha + 2$ and $r^{\prime}$ is conjugate to $r$. 
In the proof of Proposition \ref{2.1}, we showed that 
\begin{equation}
\||\Psi_{1}|^{\alpha} -
|\Psi_{2}|^{\alpha}|\|_{L^{r^{\prime}}} \leq {\rm const} 
\|\Psi_{1} - \Psi_{2} \|_{H^{1}_{0}}
\label{keyub1}
\end{equation}
uniformly in $t$.
The estimation of the first term in (\ref{first}) requires the 
pointwise inequality,
$$
\alpha \phi_{\epsilon} \ast  
(|\;|\Psi_{1}|^{\alpha -1}\partial \Psi_{1}/\partial x_{j}
- |\Psi_{2}|^{\alpha -1}\partial \Psi_{2}/\partial x_{j}\; |)| \Psi_{1}|
\leq
$$
$$
\alpha \phi_{\epsilon} \ast  
(||\Psi_{1}|^{\alpha -1} -
|\Psi_{2}|^{\alpha -1}|
|\partial \Psi_{1}/\partial x_{j}|) \; |\Psi_{1}|
+
$$
\begin{equation}
\alpha \phi_{\epsilon} \ast  
(|\partial \Psi_{1}/\partial x_{j}
- \partial \Psi_{2}/\partial x_{j}| \; 
|\Psi_{2}|^{\alpha -1}|) \:|\Psi_{1}|.
\label{keyub2}
\end{equation}
For both terms on the rhs of (\ref{keyub2}), we demonstrate that the terms
being smoothed are uniformly in $L^{1}$.  
The resulting convolution yields a function in every $L^{p}$
space, with bound
uniformly in $t$. The $L^{2}$ estimation of the product with 
$|\Psi_{1}|$ is then estimated by H\"{o}lder's
inequality. We now present the details. 

The first term on the rhs of (\ref{keyub2}) 
requires a case 
distinction ($\alpha = 1$ is trivial):
$$
1< \alpha < 2; \; 2 \leq \alpha < 4.
$$
For $2 \leq \alpha < 4$, we 
use (\ref{estPhi}) so that, pointwise, we have
$$
||\Psi_{1}|^{\alpha -1} -
|\Psi_{2}|^{\alpha -1}|
|\partial \Psi_{1}/\partial x_{j}| \leq 
(\alpha - 1) \max(|\Psi_{1}|, |\Psi_{2}|)^{\alpha - 2} ||\Psi_{1}| -
|\Psi_{2}|| \; 
|\partial \Psi_{1}/\partial x_{j}|.
$$
The application of H\"{o}lder's inequality, with indices, $1/3, 1/6, 1/2$,
resp.\, together with Sobolev's inequality, gives the desired estimate.

{\it For $1 < \alpha < 2$, this upper bound does not hold in general,
hence this interval is excluded.}

For the estimation of the second term in
(\ref{keyub2}), 
the $L^{1}$
norm of the term being smoothed is estimated by the Cauchy-Schwarz
inequality. Thus, the product of the smoothed term and $|\Psi_{1}|$ can
again be estimated in $L^{2}$ 
by the H\"{o}lder inequality.
The upper bound of the rhs of (\ref{keyub1}) is obtained. 

We now estimate the derivative of the second term in (\ref{algstep}). A
pointwise upper bound for the partial derivative is given by
$$
\alpha \phi_{\epsilon} \ast (|\Psi_{2}|^{\alpha -1}|\partial
\Psi_{2}/\partial x_{j}|)
| \Psi_{1} - \Psi_{2}| +
\phi_{\epsilon} \ast |\Psi_{2}|^{\alpha}
(|\partial \Psi_{1}/\partial x_{j} - \partial  \Psi_{2}/\partial x_{j}|).
$$

The estimation of the first of these two terms follows the previous pattern:
determination of uniform $L^{1}$ bounds for the functions being smoothed,
followed by the $L^{2}$ product estimation via the H\"{o}lder inequality. 
The upper bound of the rhs of (\ref{keyub1}) is directly obtained. 
For the second of these two terms, notice that the convolution factor is
uniformly $L^{\infty}$ by Young's convolution inequality. 
The conclusion is immediate.

A typical $kth$ term for
the partial derivative 
with resp.\ to $x_{j}$ of the Coulomb potential is bounded above pointwise:
$$
|c_{k}|\; \phi_{\epsilon} \ast (|\cdot - x_{k}|^{-2}) |\Psi_{1} - \Psi_{2}|
+ 
|c_{k}| \; \phi_{\epsilon} \ast (|\cdot - x_{k}|^{-1}) 
|\partial \Psi_{1}/ \partial x_{j} - \partial \Psi_{2}/ \partial x_{j}|.
$$ 
H\"{o}lder's inequality implies the bound for the first term. The second
term is directly estimated, since the convolution factor is uniformly
$L^{\infty}$.
This completes the proof of the proposition.
\end{proof}
\begin{corollary}
The uniqueness for the quantum corrected model holds for 
the potentials introduced in this article, except for the exclusion $1 <
\alpha < 2$ in the LDA potential. No further boundary regularity is
required.
\end{corollary}
\begin{proof}
Since the norm is weakly lower semicontinuous, the estimate is transferred
to $\|\Psi_{1} - \Psi_{2}\|$. 
Gronwall's inequality implies the result.
\end{proof}
The uniqueness result permits a useful convergence result for the
smoothing `sequence'.
\begin{corollary}
\label{full}
We assume the conditions of the uniqueness theorem.
Suppose that $\epsilon_{n}$ is any positive sequence of real numbers
convergent to zero. Then the sequence $\Psi_{\epsilon_{n}}$, satisfying
Proposition 
\ref{2.1}, converges in the norm of $C(J; H^{1}_{0}(\Omega)) \cap
C^{1}(J; H^{-1}(\Omega))$ to the unique solution $\Psi$ defined in Theorem
\ref{central}.
\end{corollary}
\begin{proof}
We use the elementary fact that, if every subsequence has a further
subsequence converging to a unique limit, then the entire sequence
converges to that unique limit. 
The first part of the proof of Theorem 
\ref{central2} demonstrates subsequential
convergence in 
$C(J; H^{1}_{0}(\Omega))$. The arguments leading to (\ref{deducefrom}) 
demonstrate convergence in
$C^{1}(J; H^{-1}(\Omega))$.

\end{proof}
\section{Summary Remarks}
We have formulated a model within the framework of time dependent density
functional theory. It is a closed system model, posed on a
bounded domain in ${\mathbb R}^{3}$ with homogeneous boundary
conditions.
The novelty of the article lies
in the flexibility of the choice of potentials. In addition to the Hartree
potential and a given external potential, 
we permit Coulomb potentials with fixed ionic point masses, a
time-history potential, and the 
local density approximation (LDA), which is
typically used in simulation. 
We have obtained existence and 
uniqueness for this model on a bounded domain in ${\mathbb R}^{3}$ and
a given finite time interval. 
The growth of the LDA term,
in terms of the exponent $\alpha$,  cannot be
modified for the methods of this article to apply. We have selected the
form here, because of its wide usage in the literature. 
Finally, Corollary \ref{full} assumes significance because the smoothed
solutions can be obtained via the evolution operator, and its
approximations (see the cited references).

We note finally, that the case of periodic boundary conditions frequently
occurs in applications. It is a topic of future study.

\appendix
\section{Notation and Norms}
\label{appendixA}
We employ complex Hilbert spaces 
in this article. 
$$
L^{2}(\Omega) = \{f = (f_{1}, \dots, f_{N})^{T}: |f_{j}|^{2} \;
\mbox{is integrable on} \; \Omega \}.
(f,g)_{L^{2}}=\sum_{j=1}^{N}\int_{\Omega}f_{j}(x){\overline {g_{j}(x)}} \; dx.
$$
However, $\int_{\Omega} fg$ is interpreted as 
$$
\sum_{j=1}^{N} \int_{\Omega} f_{j} g_{j} \;dx.
$$
For $f \in L^{2}$, as just defined, if each component $f_{j}$ satisfies
$
f_{j} \in H^{1}_{0}(\Omega; {\mathcal C}), 
$
we write $f \in H^{1}_{0}(\Omega; {\mathcal C}^{N})$, or simply, 
$f \in H^{1}_{0}(\Omega)$.
The inner product in $H^{1}_{0}$ is 
$$
(f,g)_{H^{1}_{0}}=
(f,g)_{L^{2}}+\sum_{j=1}^{N}\int_{\Omega} 
\nabla f_{j}(x) \cdotp {\overline {\nabla g_{j}(x)}} \; dx.
$$
$\int_{\Omega} \nabla f \cdotp \nabla g$ is interpreted as
$$
\sum_{j=1}^{N}\int_{\Omega} 
\nabla f_{j}(x) \cdotp \nabla g_{j}(x) \; dx.
$$
Finally, $H^{-1}$ is defined as the dual of $H^{1}_{0}$, and its
properties are discussed at length in \cite{Adams}. 
The Banach space $C(J; H^{1}_{0})$ is defined in the traditional manner:
$$
C(J; H^{1}_{0}) = \{u:J \mapsto H^{1}_{0}: u(\cdotp) \mbox{is
continuous}\}, \;
\|u\|_{C(J; H^{1}_{0}} = \sup_{t \in J} \|u(t)\|_{H^{1}_{0}}.
$$
\begin{itemize}
\item
Since $\Omega$ is assumed to be a bounded Lipschitz
domain,
the standard Sobolev embedding theorems for
$H^{1}_{0}(\Omega)$ hold, relative to $L^{p}(\Omega)$ \cite{Adams}.
\end{itemize}

\section{Subsequential Convergence for Bounded Families}
In section \ref{Convsub}, we applied two basic compactness results,
taken from
\cite{Caz} and \cite{Simon}. 
Here, we quote the underlying results for the reader's convenience. 
The first is cited from \cite[Proposition 1.3.14(i,iii)]{Caz}.
\begin{proposition}[Cazenave] 
\label{B1}
Let $I$ be a bounded interval
of ${\mathbb R}$, let $m$ be a nonnegative integer, let $\Omega$ be an
open subset of ${\mathbb R}^{N}$, and let $(f_{n})_{n \in {\mathbb N}}$
be a bounded sequence of $L^{\infty}(I; H^{1}_{0}(\Omega)) \cap
W^{1, \infty}(I; H^{-m}(\Omega))$. 

(1) Then there exist 
$(f_{n_{k}})_{k \in {\mathbb N}}$ and 
$f \in 
L^{\infty}(I; H^{1}_{0}(\Omega)) \cap
W^{1, \infty}(I; H^{-m}(\Omega))$ such that
\begin{equation*}
\forall t \in {\bar I}, \;  f_{n_{k}}(t) \rightharpoonup f(t), k
\rightarrow \infty, \; \mbox{in} \; H^{1}_{0}(\Omega). 
\end{equation*}
(2) If 
$(f_{n})_{n \in {\mathbb N}} \subset C({\bar I}; H^{1}_{0}(\Omega))$
and $\|f_{n_{k}}(t)\|_{H^{1}} \rightarrow
\|f(t)\|_{H^{1}}$ uniformly  on $I$, then $f \in C({\bar I};
H^{1}_{0}(\Omega))$ and 
$$
f_{n_{k}} \rightarrow f \; \mbox{in} \; 
 C({\bar I}; H^{1}_{0}(\Omega)).
$$ 
\end{proposition}

The next result is cited from \cite[Theorem 2.3.14]{Simon}. It is a
generalized Arzela-Ascoli theorem.
\begin{proposition}[Simon]
\label{B2}
Let $X$ be a separable metric space and $Y$ a complete metric space, with
$C \subset Y$ compact. Let ${\mathcal F}$ be a family of uniformly
equicontinuous functions from $X$ to $Y$ with Range$(f) \subset C$
for every $f \in {\mathcal F}$. Then any sequence in ${\mathcal F}$ has a
subsequence converging at each $x \in X$. If $X$ is compact, then 
${\mathcal F}$ is precompact in the uniform topology.
\end{proposition} 

\medskip

{\it Acknowledgement}: The author thanks Dr.\ Gabriele Ciaramella for
extremely helpful comments regarding the manuscript, leading to improved
accuracy and exposition.

\end{document}